\documentclass{amsart}

\usepackage{amsmath}
\usepackage{amsthm}
\usepackage{amssymb}
\usepackage{amsfonts}
\usepackage{enumitem} 
\usepackage[bookmarks=true,hyperindex,pdftex,colorlinks,citecolor=blue]{hyperref}
\usepackage{bbm}
\usepackage[all]{xy}
\usepackage{dsfont}

\DeclareMathOperator{\lspan}{span}
\DeclareMathOperator{\supp}{supp}
\DeclareMathOperator{\rad}{rad}
\DeclareMathOperator{\diam}{diam}
\DeclareMathOperator{\Lip}{Lip}
\DeclareMathOperator{\lip}{lip}

\newcommand{\N}{\mathbb{N}}             
\newcommand{\Natural}{\mathbb N}        
\newcommand{\R}{\mathbb{R}}             
\newcommand{\Real}{\mathbb R}           

\newcommand{\abs}[1]{\left|{#1}\right|}                     
\newcommand{\set}[1]{\left\{{#1}\right\}}                   
\newcommand{\norm}[1]{\left\|{#1}\right\|}                  
\newcommand{\ball}[1]{B_{{#1}}}                             

\newcommand{\sphere}[1]{S_{{#1}}}                           
\newcommand{\duality}[1]{\left<{#1}\right>}                 

\newcommand{\wconv}{\stackrel{w}{\rightarrow}}              

\newcommand{\lipfree}[1]{\mathcal{F}({#1})}                 
\newcommand{\F}[1]{\mathcal{F}(#1)}                                
\newcommand{\Free}{\mathcal{F}}                             
\newcommand{\lipnorm}[1]{\norm{#1}_L}                       

\newcommand{\restricted}{\mathord{\upharpoonright}}

\def\<{\langle}
\def\>{\rangle}
\newcommand{\ep}{\varepsilon}

\theoremstyle{plain}
\newtheorem{theorem}{Theorem}[section]
\newtheorem{lemma}[theorem]{Lemma}
\newtheorem{corollary}[theorem]{Corollary}
\newtheorem{proposition}[theorem]{Proposition}

\newtheorem*{claim*}{Claim}

\theoremstyle{definition}
\newtheorem*{definition*}{Definition}
\newtheorem{definition}[theorem]{Definition}
\newtheorem{example}[theorem]{Example}

\begin{document}
\title{Compact reduction in Lipschitz free spaces}

\author[R. J. Aliaga]{Ram\'on J. Aliaga}
\address[R. J. Aliaga]{Universitat Polit\`ecnica de Val\`encia, Camino de Vera S/N, 46022 Valencia, Spain} 
\email{raalva@upvnet.upv.es}

\author[C. No\^us]{Camille No\^us}
\address[C. No\^us]{Laboratoire Cogitamus} 
\email{camille.nous@cogitamus.fr}

\author[C. Petitjean]{Colin Petitjean}
\address[C. Petitjean]{LAMA, Univ Gustave Eiffel, UPEM, Univ Paris Est Creteil, CNRS, F--77447, Marne-la-Vall\'ee, France} 
\email{colin.petitjean@univ-eiffel.fr}

\author[A. Proch\'azka]{Anton\'in Proch\'azka}
\address[A. Proch\'azka]{Laboratoire de Math\'ematiques de Besan\c con  UMR 6623,
Universit\'e Bourgogne Franche-Comt\'e, CNRS, F-25000, Besan\c con, France} 
\email{antonin.prochazka@univ-fcomte.fr}

\begin{abstract}
We prove a general principle satisfied by weakly precompact sets of Lipschitz-free spaces.
By this principle, 
certain infinite dimensional phenomena in Lipschitz-free spaces over general metric spaces may be reduced to the same phenomena in free spaces over their compact subsets.
As easy consequences we derive several new and some known results.
The main new results are: $\Free(X)$ is weakly sequentially complete for every superreflexive Banach space $X$, and $\Free(M)$ has the Schur property and the approximation property for every scattered complete metric space $M$.
\end{abstract}


\maketitle

\section{Introduction}

For a metric space $M$ with a distinguished base point $0 \in M$ (commonly called a \emph{pointed metric space}), the \emph{Lipschitz free space} (for brevity just \emph{free space} in the sequel) $\Free(M)$ is the norm-closed linear span of the evaluation functionals, i.e. of the set $\set{\delta(x): x \in M}$ in the space $\Lip_0(M)^*$ where $\delta(x)\colon f\mapsto f(x)$.
Here the Banach space
$$
\Lip_0(M)=\set{f \in \Real^M: f \mbox{ Lipschitz}, f(0)=0}
$$
is equipped with the norm \[\lipnorm{f}:=\sup\set{\frac{f(x)-f(y)}{d(x,y)}:x \neq y}.\]
It is well known that $\Free(M)$ is an isometric predual of $\Lip_0(M)$.
Another main feature of the free spaces is that every Lipschitz map $f:M \to N$ which fixes the zero induces a linear map $\hat{f}:\Free(M)\to \Free(N)$ such that $\delta_N \circ f=\hat{f}\circ \delta_M$ and $\|\hat{f}\|=\norm{f}_L$.
For a quick proof of these facts and some other basic properties we refer the reader to the paper \cite{CDW_2016}.

The study of isomorphic and isometric properties of free spaces has been a very active research area recently where many deep theorems have been proved but many basic questions are left hopelessly open. 
In this paper we focus on isomorphic properties of free spaces. 
Our starting point is an innocent looking lemma about weakly null sequences in free spaces in the fundamental paper by Nigel Kalton~\cite{Kalton04}. 
In Nigel's words: ``weakly-null sequences [in $\Free(M)$] are almost supported on `small' sets [of $M$]''.
By `small', it is meant unions of finite collections of small-radius balls in the metric space $M$; for the precise statement see Definition~\ref{def:KaltonProperty}.
Here we observe that, when $M$ is complete, this lemma can be bootstrapped to obtain a more user-friendly conclusion: weakly-null sequences in $\Free(M)$ are almost supported on compact subsets of $M$. 
In fact, this conclusion, which we call \emph{tightness} (see Definition~\ref{def:tightness}), is not only true for weakly-null sequences but also for all so called \emph{weakly precompact} sets (see Definition~\ref{def:WeakPrecompact}).
In light of (a bit more advanced version of) this principle, a number of intriguing questions obtain straightforward answers.

Thus, we have several ``compact reduction'' results: $\Free(M)$ is weakly sequentially complete (WSC), resp. Schur, resp. $\ell_1$-saturated, if and only if $\Free(K)$ is for every compact subset $K$ of $M$. 
Similar results are obtained for the approximation property and the Dunford-Pettis property.
Also if a Banach space $X$ does not contain a copy of $\ell_1$, then $\Free(M)$ contains a copy of $X$ if and only if $\Free(K)$ contains a copy of $X$ for some compact $K \subset M$.
Shortly before the original announcement of our results, a preprint was published by Gartland \cite{Gartland_2020} where restricted versions of some of our statements (in particular Corollaries \ref{c:CompactReductionSchur} and \ref{c:FCountableCompletIsSchur}) are proved using a related technique. It was his suggestion that we try to extend our compact reduction principle to the approximation property, for which we are grateful.

Combining these compact-reduction results with known results about free spaces 
we obtain that $\Free(X)$ is WSC for every superreflexive Banach space $X$ (which answers a question posed 
in~\cite{CDW_2016} by C\'uth, Doucha and Wojtaszczyk, but we humbly acknowledge that the heavy lifting was done by Kochanek and Perneck\'a who solved the compact case in~\cite{KP_2018}). 
Notice that in particular one obtains that $\Free(c_0)$ is not isomorphic to $\Free(\ell_p)$ for $1<p<\infty$.
Up to our knowledge, these are 
the first examples of \emph{classical} infinite dimensional separable Banach spaces whose free spaces are not isomorphic.
We also obtain that $\Free(M)$ has the Schur property and the approximation property for every scattered complete metric space. 
Further we (re)prove that non-separable Asplund spaces, WCG spaces and $\ell_\infty$ cannot be isomorphic to a subspace of any free space. 
All of this is detailed in Section~\ref{s:Tightness}.
The proof of our user-friendly Kalton's lemma is to be found in Section~\ref{s:Proof}. 
It depends heavily on the notions of support and multiplication operator developed in~\cite{AliagaPernecka,APPP_2020}. 
Finally, in Section~\ref{s:SchurExamples} we slightly improve a sufficient condition for the Schur property of $\Free(M)$ coming from~\cite{Petitjean} and disprove (twice!) the conjecture about necessity of such condition.

\subsection{Notation}

Let us now introduce the notation that will be used throughout this paper. For a Banach space $X$, we will write $B_X$ for its closed unit ball and $S_X$ for its unit sphere.
As usual, $X^*$ denotes the topological dual of $X$ and $\duality{x^\ast,x}$ will stand for the evaluation of $x^* \in X^*$ at $x \in X$. We will write $w = \sigma(X,X^*)$ for the weak topology in $X$ and $w^* = \sigma(X^*,X)$ for the weak$^*$ topology in $X^*$.

The letter $M$ will denote a \emph{complete} pointed metric space
with metric $d$ and base point $0$. 
The choice of the base point will be irrelevant to our results since, as is well known, free spaces over the same metric space but with different base points are isometrically isomorphic. We recall that if $N \subset M$ and $0\in N$, then $\Free(N)$ can be canonically isometrically identified with the subspace $\overline{\lspan}\set{\delta(x):x\in N}$ of $\Free(M)$. This is due to well known McShane-Whitney theorem,
according to which every real-valued Lipschitz function on $N$ can be extended to $M$ with the same Lipschitz constant.
Further, $B(p, r)$ will stand for the closed ball of radius $r$
around $p \in M$ while for $A \subset M$ and $\delta>0$, we will write
\begin{eqnarray*}
d(p,A) &=& \inf\{d(p,x) \; : \; x \in A\} \\
{[A]}_{\delta} &=&  \{p \in M \; : \;  d(p,A)\leq \delta \}.
\end{eqnarray*}
We will also use the notation
\begin{eqnarray*}
\rad(A) &=& \sup\{d(0,x) \; : \; x \in A\} \\
\diam(A) &=& \sup\{d(x,y) \; : \; x,y \in A\}.
\end{eqnarray*}
These two last quantities will be called the radius of $A$ and the diameter of $A$, respectively. 
Next, for any set $B\subset M$ we define the Kuratowski measure of noncompactness $\alpha(B)$ as the infimum of the numbers $r>0$ such that $B$ admits a finite covering by sets of diameter smaller than $r$.
Let us recall that Kuratowski's theorem (see \cite{Kuratowski}) states that if $(B_n)_n$ is a decreasing sequence of nonempty, closed subsets of $M$ such that $\lim\limits_{n \to \infty} \alpha(B_n) = 0$, then the intersection $B$ of all $B_n$ is nonempty and compact.

Finally, let us recall some notions from \cite{AliagaPernecka,APPP_2020} that will be used throughout the proof of our main theorem. Given any Lipschitz function $h$ on $M$ with bounded support, the pointwise product $fh$ belongs to $\Lip_0(M)$ for any $f\in\Lip_0(M)$, and its support is contained in $\supp(h)$. In fact, for any $N\supset\supp(h)$ the mapping $T_h : \Lip_0(N) \to \Lip_0(M)$ defined by 
\begin{equation}
\label{mult_operator}
T_h(f)(x)=
\begin{cases}
h(x)f(x)& \mbox{if }x \in N \\
0& \mbox{otherwise}
\end{cases}
\end{equation}
is a $w^*$-to-$w^*$ continuous linear operator, whose norm is bounded by
\begin{equation}
\label{mult_operator_norm}
\norm{T_h}\leq\norm{h}_\infty+\rad(\supp(h))\lipnorm{h} .
\end{equation}
Therefore its adjoint operator $T_h^\ast$ takes $\lipfree{M}$ into $\lipfree{N}$. See \cite[Lemma 2.3]{APPP_2020} for the detailed proof of these facts.
Also, recall that for each $\mu\in\lipfree{M}$ we can define its \emph{support} as the set
$$
\supp(\mu)=\bigcap\set{K\subset M: \text{$K$ is closed and $\mu\in\lipfree{K}$}} .
$$
It satisfies $\mu\in\lipfree{\supp(\mu)}$, and moreover $\mu\in\lipfree{K}$ if and only if $K\supset\supp(\mu)$. This notion coincides with the usual one for finite linear combinations of evaluation functionals, i.e. finitely supported elements of $\lipfree{M}$. We refer to Section 2 of \cite{APPP_2020} for proofs of this and additional properties.

\section{Tightness of weakly precompact sets and applications}\label{s:Tightness}

The following definition is somewhat reminiscent of the concept of tightness for subsets of Borel measures on complete metric spaces (see \cite{Bogachev}).  
\begin{definition}\label{def:tightness}
We will say that a set $W\subset \Free(M)$ is \emph{tight} if for every $\varepsilon>0$
there exists a compact $K\subset M$ such that 
\[
W \subset \Free(K) + \varepsilon\ball{\Free(M)}.
\]
\end{definition}

For the next one, let us recall that a sequence $(x_n)_n$ in a Banach space $X$ is \emph{weakly Cauchy} if the sequence $(\langle x^* , x_n \rangle)_n$ is convergent for every $x^* \in X^*$.

\begin{definition} \label{def:WeakPrecompact}
Recall that a subset $W$ of a Banach space $X$ is called \emph{weakly precompact}\footnote{Do not confuse with \emph{relatively weakly compact sets}!} if every sequence $(x_n)_n \subset W$ admits a weakly Cauchy subsequence.
Equivalently, by virtue of Rosenthal's $\ell_1$-theorem,  $W$ is weakly precompact if it is bounded and no sequence in $W$ is equivalent to the unit vector basis of $\ell_1$.
\end{definition}
Our main technical result is the following.
\begin{theorem}\label{thm:WeaklyPrecompactIsTight}
Let $M$ be a complete metric space.
Let $W \subset \Free(M)$ be weakly precompact.
Then $W$ is tight.

More precisely, for every $\varepsilon>0$ there exist a compact $K\subset M$ and a  linear mapping $T:\lspan(W) \to \Free(K)$ such that
\begin{itemize}
    \item \ $\norm{\mu-T\mu}\leq \varepsilon$ for all $\mu\in W$, and
    \item \ there is a sequence of bounded linear operators $T_k:\Free(M)\rightarrow \Free(M)$ such that $T_k\rightarrow T$ uniformly on $W$.
\end{itemize}
\end{theorem}

We postpone the proof of this theorem until
Section~\ref{s:Proof} in order to discuss its most important consequences first.

\subsection{Compact reduction for weak sequential completeness}
Recall that a Banach space $X$ is called \emph{weakly sequentially complete}  (WSC) if every weakly Cauchy sequence in $X$ is weakly convergent.
\begin{corollary}\label{c:CompactReductionWSC}
Let $M$ be a complete metric space. 
Then $\Free(M)$ is WSC if and only if $\Free(K)$ is WSC for every compact $K\subset M$.
\end{corollary}

\begin{proof}
\ Since the WSC property
passes to subspaces we only need to prove the sufficiency.
Let $(\mu_n)_n$ be weakly Cauchy in $\Free(M)$. 
Let $\mu$ be its $w^*$-limit in $\Free(M)^{**}$.
We set $W:=\set{\mu_n:n \in \Natural}$ and let $\varepsilon>0$ be arbitrary.
Let $K$ and $T$ be as in Theorem~\ref{thm:WeaklyPrecompactIsTight} with 
$\norm{T\mu_n-\mu_n}\leq \varepsilon$ for all $n$. 
We know that there are bounded linear operators $T_k:\Free(M)\to \Free(M)$ such that $T_k\restricted_W \to T\restricted_W$ uniformly.
So for every fixed $f \in \Lip_0(M)$ we have
\[
 \sup_{n\in \Natural} \abs{ \duality{f,T_k\mu_n-T\mu_n}} \to 0 \quad\mbox{as } k \to \infty.
\]
Moreover, every sequence $(\duality{f,T_k\mu_n})_n$ is Cauchy (for fixed $k$).
So the usual exchange-of-limits argument gives that the sequence $(\duality{f,T\mu_n})_n$ is Cauchy, too.
Hence $(T\mu_n)_n$ is weakly Cauchy.
By assumption, $\Free(K)$ is WSC so there exists $\lambda \in \Free(K)$ such that $T\mu_n\wconv\lambda$.
By the weak lower semicontinuity of the norm we have $\norm{\lambda-\mu}\leq \varepsilon$ and so $\mu \in \Free(M)$ since $\varepsilon$ was arbitrary.
\end{proof}

In the deep paper \cite{KP_2018}, Kochanek and Perneck\'a prove that if $M$ is a compact subset of a superreflexive Banach space, then the Lipschitz-free space $\Free(M)$ is weakly sequentially complete.
So Corollary~\ref{c:CompactReductionWSC} implies immediately:

\begin{corollary} \label{corFreeSupRwsc}
If $X$ is a superreflexive Banach space then $\Free(X)$ is WSC.
\end{corollary}

In particular, the space $\Free(\ell_2)$ is WSC. 
This provides a negative answer to Question~3 posed by C\'uth, Doucha and Wojtaszczyk in~\cite{CDW_2016}.

\subsection{Compact reduction for the Schur property}
A Banach space has the \emph{Schur property} if every weakly-null sequence is also norm-convergent to $0$.
\begin{corollary}\label{c:CompactReductionSchur}
Let $M$ be complete. 
The free space $\Free(M)$ has the Schur property if and only if $\Free(K)$ has the Schur property for every compact $K\subset M$.
\end{corollary}

The proof is almost identical to the proof of Corollary~\ref{c:CompactReductionWSC} and is left 
to the reader.

\begin{corollary} \label{c:FCountableCompletIsSchur}
Let $M$ be countable and complete metric space. 
Then $\Free(M)$ has the Schur property.
More generally, let $M$ be a scattered complete metric space. 
Then $\Free(M)$ has the Schur property.
\end{corollary}

Notice that the above corollary applies in particular to complete metric spaces which are topologically discrete. 
The Schur property for free spaces of such spaces was, up to our knowledge, not known. See Section~\ref{s:SchurExamples} for more information on this subject.

\begin{proof}[Proof of Corollary \ref{c:FCountableCompletIsSchur}]
\ In a countable complete metric space every compact is clearly countable so the first claim follows from Corollary~\ref{c:CompactReductionSchur} and \cite[Theorem 3.1]{HLP_2016}, which states that $\lipfree{M}$ is Schur if $M$ is countable and compact (see also~\cite{Petitjean}). 

In a scattered metric space  (i.e. without a perfect part) every subset is scattered.
But the only scattered and compact metric spaces are countable compacts (see Lemma~VI.8.2 in~\cite{DGZ}).
So again we conclude by Corollary~\ref{c:CompactReductionSchur} and~\cite[Theorem 3.1]{HLP_2016}.

Alternatively, the first claim also follows from the second one using the fact that perfect sets are uncountable.
\end{proof}

\subsection{Compact reduction for the approximation property}

We recall that a Banach space $X$ has the \emph{approximation property (AP)} if for every $\ep > 0$ and every compact set $W \subset X$ there exists a finite-rank bounded operator $S : X \to X$
such that $\|Sx-x\| \leq \ep$ for every $x \in W$.

We are grateful to Chris Gartland for suggesting the possibility of using our main theorem for proving the following result.

\begin{corollary}\label{c:CompactReductionAP}
Let $M$ be complete. Then $\Free(M)$ has the AP if and only if for every compact $K\subset M$ there is a subset $B\subset M$ such that $K\subset B$ and $\Free(B)$ has the AP. 
\end{corollary} 
\begin{proof}
\ One direction is clear, let us prove the other one.
Let $W \subset \Free(M)$ be a norm-compact set.
Let $\varepsilon>0$, we will find a finite rank operator $S:\Free(M)\to \Free(M)$ such that $\sup_{w\in W}\norm{Sw-w}\leq 3\varepsilon$.
So let $K$, $T:W \to \Free(K)$ and $(T_n)$ be given by Theorem~\ref{thm:WeaklyPrecompactIsTight}.

Since $T\restricted_W$ is the uniform limit of continuous maps, it is continuous. Thus $T(W) \subset \Free(K)$ is norm compact.

By assumption there exists $B\subset M$ such that $K\subset B$ and $\Free(B)$ has the AP. 
So there exists a bounded finite-rank operator $H:\Free(B) \to \Free(B)$ such that $\sup_{w\in W}\norm{HTw-Tw}\leq \varepsilon$.
Let $\widetilde{H}:\Free(M)\to \Free(M)$ be a Hahn-Banach finite-rank extension of $H$.
We also know by the uniform convergence of $(T_n)$ to $T$ on $W$ that  there is $n$ such that $\sup_{w \in W}\norm{T_nw-Tw}\leq \varepsilon \norm{\widetilde{H}}^{-1}$.
We set $S=\widetilde{H}T_n$ which is bounded and has finite rank. For every $w\in W$ we have
\[
\begin{aligned}
\norm{Sw-w}&\leq \norm{\widetilde{H}T_nw-\widetilde{H}Tw}+\norm{\widetilde{H}Tw-Tw}+\norm{Tw-w}\\ 
&\leq \norm{\widetilde{H}}\norm{T_nw-Tw}+\varepsilon+\varepsilon\leq 3\varepsilon
\end{aligned}
\]
which we wanted to prove.
\end{proof}

Notice that the difference between the statement of
Corollary~\ref{c:CompactReductionWSC} (or Corollary~\ref{c:CompactReductionSchur}) and our last corollary is necessary since the AP does not pass to subspaces. For instance, consider a Banach space $X \subset c_0$ failing the AP
such as Enflo's example \cite{Enflo_AP}.
It follows from \cite[Theorem 5.3]{GoKa_2003} that $\Free(c_0)$ has the AP. However, there is a compact convex subset $K \subset X$ such that $X$ is isometric to a 1-complemented subspace of $\Free(K)$ and thus $\Free(K)$ fails the AP (see \cite[Theorem 4]{GodefroyOzawa}).

Since free spaces over countable compact spaces have the AP \cite[Theorem 3.1]{Dalet_2015}, we obtain in particular:

\begin{corollary}
Let $M$ be countable and complete metric space. 
Then $\Free(M)$ has the AP.
More generally, let $M$ be a scattered complete metric space. 
Then $\Free(M)$ has the AP.
\end{corollary}
Notice that the above corollary gives a  positive answer to the Banach space case of  \cite[Question 6.3]{AACD20}.

\subsection{Compact reduction for the Dunford--Pettis property}

A Banach space $X$ is said to have the \emph{Dunford--Pettis property (DPP)} if for every sequence $(x_n)_n$ in $X$ converging weakly to 0 and every sequence $(x_n^*)_n$ in $X^*$ converging weakly to $0$, the sequence of scalars $(x_n^*(x_n))_n$ converges to 0. Note that this property does not pass to subspaces.

\begin{corollary}\label{c:CompactReductionDPP}
Let $M$ be complete. 
Then $\Free(M)$ has the DPP if and only if for every compact $K\subset M$ there is a subset $B\subset M$ such that $K\subset B$ and $\Free(B)$ has the DPP.
\end{corollary}

\begin{proof}
\ The direction ``$\implies$" is trivial, so let us prove the other one. Let $(\mu_n)_n \subset \Free(M)$  and $(f_n)_n \subset \Lip_0(M)$ be weakly-null (so bounded by some constant $C>0$).
We set $W:=\set{\mu_n:n \in \Natural}$ and let $\varepsilon>0$ be arbitrary.
Let $K$ and $T$ be as in Theorem~\ref{thm:WeaklyPrecompactIsTight} with 
$\norm{T\mu_n-\mu_n}\leq \varepsilon$. 
In the same way as in the proof of Corollary~\ref{c:CompactReductionWSC}, we obtain that $(T \mu_n)_n$ is weakly-null.
By assumption, there exists $B \subset M$ such that $K \subset B$ and $\Free(B)$ has the DPP. Since $(T \mu_n)_n$ is weakly-null in $\Free(K)$ and thanks to the canonical isometric identification $\Free(K) \subset \Free(B)$, $(T\mu_n)_n$ is weakly null in $\Free(B)$.
Moreover the sequence $(f_n\restricted_{B})_n \subset \Lip_0(B)$ is also weakly-null. Therefore $\lim_n \duality{f_n,T\mu_n}=0$. Now let $n_0 \in \N$ such that $\abs{\duality{f_n,T\mu_n}} \leq \varepsilon$ whenever $n \geq n_0$. Using the triangle inequality we deduce that for every $n \geq n_0$:
$$ \abs{\duality{f_n,\mu_n}} \leq \abs{\duality{f_n,\mu_n-T\mu_n}}+\abs{\duality{f_n,T\mu_n}} \leq \|f_n\|_L \|\mu_n-T\mu_n\| + \varepsilon  \leq (C+1)\varepsilon.$$
Since $\varepsilon>0$ was arbitrary, this yields the conclusion. 
\end{proof}

\subsection{Compact reduction for copies of spaces not containing \texorpdfstring{$\ell_1$}{l1}}

\begin{theorem}\label{thm:CompactReductionNonEll1}
Let $M$ be a complete metric space. 
Let $X \subseteq \Free(M)$ be a closed subspace which does not contain an isomorphic copy of $\ell_1$.
Then there exists a compact $K\subset M$ such that $X$ is isomorphic to a subspace of $\Free(K)$.
\end{theorem}
\begin{proof}
\ We apply Theorem~\ref{thm:WeaklyPrecompactIsTight} for $W=B_X$ and $\varepsilon<\frac12$.
Notice that $\lspan(B_X)=X$.
Since $T\restricted_{B_X}$ is a uniform limit of bounded operators, it follows that $T\restricted_X$ is a bounded linear operator.
On the other hand we have for every $x\in \sphere{X}$ that $\norm{Tx}\geq 1-\varepsilon$.
It follows that $T:X \to \Free(K)$ is an isomorphism between $X$ and a subspace of $\Free(K)$.
\end{proof}

We get immediately the following compact-reduction result.

\begin{corollary}\label{c:CompactReductionEll1Saturated}
Let $M$ be an infinite complete metric space. Assume that for every infinite compact $K \subset M$, the space $\Free(K)$ is $\ell_1$-saturated (i.e. every infinite-dimensional subspace of $\Free(K)$ contains an isomorph of $\ell_1$). 
Then $\Free(M)$ is $\ell_1$-saturated.
\end{corollary}

\subsection{Consequences for non-separable free spaces}

The next result, which  follows immediately from Theorem~\ref{thm:WeaklyPrecompactIsTight}, is known.
\begin{corollary}\label{c:WeaklyPrecompactAreSeparable}
Let $W \subset \Free(M)$ be weakly precompact. Then $W$ is separable.
\end{corollary}

This fact has been first observed for weakly compact subsets of free spaces over weakly compactly generated Banach spaces in~\cite[Proposition 4.1]{GoKa_2003} and later proved for weakly precompact subsets of free spaces over arbitrary metric spaces in~\cite[Theorem 2.1]{Kalton_2011}.

We also get easily that some well known (classes of) non-separable Banach spaces do not appear as subspaces of free spaces.
Consequently, they fail the \emph{Lipschitz lifting property} introduced by Godefroy and Kalton in \cite{GoKa_2003}. Recall that a Banach space $X$ is said to have this property if the map
$$\beta_X\colon \sum_{i=1}^n a_i \delta(x_i) \in \Free(X) \longmapsto \sum_{i=1}^n a_i x_i \in  X$$
admits a bounded linear right inverse.
It was already known that cases (ii) and (iii) below fail this property (see \cite[Theorems 4.3 and 4.6]{GoKa_2003}) but the proof here is different.

\begin{corollary}
The following non-separable Banach spaces are not isomorphic to a subspace of any free space:
\begin{itemize}
    \item[(i)] \ all non-separable spaces not containing a copy of $\ell_1$ (e.g. Asplund spaces, the dual $JT^*$ of the James tree space, etc.),
    \item[(ii)] \ non-separable weakly compactly generated (WCG) spaces,
    \item[(iii)] \ $\ell_\infty$.
\end{itemize}
In particular, all these spaces fail the Lipschitz lifting property. And also, the free spaces of all these spaces fail 
to have unique Lipschitz structure.
\end{corollary}

\begin{proof}
\ (i) follows directly from Theorem~\ref{thm:CompactReductionNonEll1}. 
The fact that $JT^*$ does not contain a copy of $\ell_1$ is well known and is proved for instance in~\cite[Corollary 3.c.7]{JamesForest}.

(ii) If $X\subseteq \Free(M)$ is non-separable and WCG, there is a weakly compact $W \subset X$ such that $\overline{\lspan}(W)=X$. It follows that $W$ is non-separable, contrary to Corollary~\ref{c:WeaklyPrecompactAreSeparable}.

(iii) $\ell_\infty$ contains an isometric copy of the dual of every separable Banach space, in particular contains $JT^*$.

Finally, if a Banach space $X$ satisfies the Lipschitz lifting property then $\Free(X)$ contains an isomorphic copy of $X$.
This proves the second-to-last statement. 

Also, let $\beta_X:\Free(X) \to X$ be the linear extension of $Id_X:X\to X$, i.e. $\beta_X\circ \delta_X=Id_X$. 
Then it is well known and easy to see that $\Free(X)$ is Lipschitz equivalent to $\ker \beta_X\oplus X$ \cite{GoKa_2003}. But, because of the first part of the corollary, $\ker \beta_X\oplus X$ cannot be linearly isomorphic to $\Free(X)$. 
This proves the last statement.
\end{proof}

\section{Proof of Theorem~\ref{thm:WeaklyPrecompactIsTight}}\label{s:Proof}

We will in fact prove a (perhaps only formally) more general theorem, and then show that Theorem~\ref{thm:WeaklyPrecompactIsTight} follows from it. For the general result we need the following notion.

\begin{definition}\label{def:KaltonProperty}
Let us say that a set $W\subset\lipfree{M}$ has \emph{Kalton's property} if it is such that for every $\varepsilon,\delta>0$ there exists a finite set $E\subset M$ such that \[W\subset\lipfree{[E]_\delta}+\varepsilon\ball{\lipfree{M}}\]
where $[E]_\delta=\set{x \in M: d(x,E)\leq \delta}$.
\end{definition}
In~\cite[Lemma 4.5]{Kalton04} Kalton proved that weakly null sequences in free spaces over bounded metric spaces satisfy this property -- hence the terminology. It is clear that tightness implies Kalton's property. The next theorem shows that they are in fact equivalent.

\begin{theorem} \label{thm:KaltonPropertyToCompact}
Let $W \subset \lipfree{M}$ have Kalton's property.
Then $W$ is tight.

More precisely, for every $\varepsilon>0$ there exist a compact $K\subset M$ and a linear map $T:\lspan(W)\to\Free(K)$ such that
\begin{itemize}
    \item \ $\norm{\mu-T\mu}\leq \varepsilon$ for all $\mu\in W$, and
    \item \ there is a sequence of bounded linear operators $T_k:\Free(M)\rightarrow\Free(M)$ such that $T_k\rightarrow T$ uniformly on $W$.
\end{itemize}
\end{theorem} 

\begin{proof}

In the proof we will use the following fact.
 
\begin{claim*}
Let $N\subset M$ and suppose that $S:\lipfree{M}\rightarrow\lipfree{N}$ is a bounded operator such that $\supp(S\mu)\subset\supp(\mu)\cap N$ for every $\mu\in\lipfree{M}$. Then the image $S(W)$ has Kalton's property (relative to $N$).
\end{claim*}

\begin{proof}[Proof of the Claim]
Let $F\subset M$ be finite and $\delta>0$.
Then there is a finite $F' \subset N$ such that $[F]_\delta \cap N \subset [F']_{2\delta} \cap N$.
Indeed, for every $x\in F$ we pick $x' \in B(x,\delta)\cap N$ if it exists and let $F'$ be the collection of all chosen $x'$ .
It follows that
$$
S(\Free([F]_\delta))\subset \Free([F]_\delta \cap N) \subset \Free([F']_{2\delta}\cap N).
$$
So if $F\subset M$ and $\delta,\varepsilon>0$ are such that $W \subset \Free([F]_\delta)+\varepsilon B_{\Free(M)}$, then
$$
S(W) \subset \Free([F']_{2\delta}\cap N)+\norm{S}\varepsilon B_{\Free(N)}
$$
and the conclusion now easily follows.
\end{proof}

Let us first prove Theorem \ref{thm:KaltonPropertyToCompact} under the additional assumption that $\diam(M)=R<\infty$. We may suppose that $\varepsilon<1$. Let $\varepsilon_0=\varepsilon$, $\delta_0=R$, and for $n\geq 1$ denote $\varepsilon_n=2^{-n}\varepsilon$ and $\delta_n=R(\frac{1}{\varepsilon_n}-2)^{-1}$. Let $W_0=W$, $K_0=M$ and $S_0$ be the identity operator on $\lipfree{M}$. 
We will inductively construct sets $K_n\subset M$ and operators $S_n:\lipfree{K_{n-1}}\rightarrow\lipfree{K_n}$ for $n\geq 1$ such that $(K_n)_n$ is decreasing, $\alpha(K_n)\leq 4\delta_n$, each $S_n$ satisfies the hypothesis of the Claim, and $\norm{\mu-S_n\mu}\leq\varepsilon_n$ for every $\mu\in W_{n-1}$, where $W_n=T_n(W)$ and
$$T_n=S_n\circ\ldots\circ S_1.$$
Suppose $K_{n-1}$ and $S_{n-1}$ have been constructed. 
Since $W_{n-1}$ has Kalton's property (with respect to the metric space $K_{n-1}$) by hypothesis and by the Claim, there exists a finite set
$E_n \subset K_{n-1}$
such that
$$
W_{n-1} \subset \lipfree{[E_n]_{\delta_n}}+\varepsilon_n^2\ball{\lipfree{K_{n-1}}}.
$$
In the above we understand $[E_n]_{\delta_n}=\set{x\in K_{n-1}:d(x,E_n)\leq \delta_n}$ and similarly for $[E_n]_{2\delta_n}$ below.
Let $K_n=[E_n]_{2\delta_n}$, which clearly satisfies $\alpha(K_n)\leq 
4\delta_n$.
By the McShane-Whitney extension theorem (plus bounding above and below by $1$ and $0$),
there is a Lipschitz function $h_n$ on $K_{n-1}$ such that $0\leq h_n\leq 1$, $\lipnorm{h_n}\leq\frac{1}{\delta_n}$ and
$$ h_n(x) = \left\{\begin{array}{ccc}
    1 & \text{if} & x\in  [E_n]_{\delta_n}, \\
    0 & \text{if} &  x \in K_{n-1}\setminus K_n.
\end{array}\right. $$
Let $T_{h_n}:\Lip_0(K_n)\rightarrow\Lip_0(K_{n-1})$ the multiplication operator given by \eqref{mult_operator} for $h=h_n$, and let $S_n:\lipfree{K_{n-1}}\rightarrow\lipfree{K_n}$ be its preadjoint. 
Note that $\norm{S_n}\leq 1+R\lipnorm{h_n}$ by \eqref{mult_operator_norm}.
Clearly $S_n$ acts as the identity on $\lipfree{[E_n]_{\delta_n}}$ and its image is contained in $\lipfree{K_n}$.
Moreover $S_n$ satisfies the hypothesis of the Claim (see \cite[remark after Proposition 2.6]{APPP_2020}).
Finally, given $\mu\in W_{n-1}$ there exists $\lambda\in\lipfree{[E_n]_{\delta_n}}$ with $\norm{\mu-\lambda}\leq\varepsilon_n^2$, so we have
\begin{align*}
\norm{\mu-S_n\mu} &\leq \norm{\mu-\lambda}+\norm{\lambda-S_n\lambda}+\norm{S_n(\lambda-\mu)} \\
&\leq (1+\norm{S_n})\varepsilon_n^2 \leq (2+R\lipnorm{h_n})\varepsilon_n^2 \leq \varepsilon_n .
\end{align*}
This completes the construction.

For every $\mu\in W$, the sequence $(T_n\mu)_n$ is Cauchy by construction, hence it converges to some $\lambda\in\lipfree{M}$. Moreover
$$
\norm{\mu-T_n\mu} \leq \sum_{k=1}^n \norm{T_{k-1}\mu-T_k\mu} \leq \sum_{k=1}^n \varepsilon_k < \varepsilon
$$
and $\norm{\mu-\lambda}\leq\varepsilon$. 
We denote $K=\bigcap_{n=1}^\infty K_n$. 
Notice that $K$ is compact by Kuratowski's theorem~\cite{Kuratowski}.
Moreover $\supp(\lambda)\subset K$.
Indeed, $\lambda \in \bigcap_{n=1}^\infty\Free(K_n)=\Free(K)$ where the equality follows from \cite[Theorem 2.1]{APPP_2020}.

Now assume $M$ is unbounded. 
Then we precede the above construction by a preliminary step as follows: 
by Kalton's property, there is a finite set $E\subset M$ such that
$$
W \subset \lipfree{[E]_1}+\frac{\varepsilon}{8}\ball{\lipfree{M}} .
$$
Let $R=2(\rad(E)+1)$ and $K_0=B(0,R)$. 
Construct a Lipschitz function $h$ on $M$ with $0\leq h\leq 1$, $h=1$ on $B(0,R/2)$, $h=0$ on $M\setminus K_0$, and $\lipnorm{h}\leq\frac{2}{R}$.
Then similarly $\norm{T_h}\leq 3$ using \eqref{mult_operator_norm} and for any $\mu\in W$ we get
$$
\norm{\mu-T_h^\ast\mu}\leq (1+\norm{T_h})\frac{\varepsilon}{8}\leq\frac{\varepsilon}{2} .
$$
Since $T_h^\ast(W)\subset\lipfree{K_0}$ also has Kalton's property by the Claim, we can now apply the first part of the proof to obtain a compact $K \subset K_0$ such that $T_h^\ast(W)\subset\lipfree{K}+\frac{\varepsilon}{2}\ball{\lipfree{M}}$ with the corresponding operators $T_k:\Free(K_0)\to \Free(K_0)$ and the limit map $T:\Free(K_0)\to \Free(K)$. 
Clearly the operators $T_k\circ T_h^\ast$ and $T\circ T_h^\ast$ satisfy the requirements in the second half of the statement.
This ends the proof.
\end{proof}

In order to prove Theorem~\ref{thm:WeaklyPrecompactIsTight} it is now enough to show that weakly precompact sets in free spaces have Kalton's property.
This is achieved separately for bounded and unbounded metric spaces in the following couple of propositions, whose arguments are inspired by Kalton's original one from \cite[Lemma 4.5]{Kalton04}.

\begin{proposition}[Bounded case] \label{prop:generalKaltonLemma}
Let $M$ be bounded and let $W \subset \Free(M)$ be weakly precompact.
Then $W$ has Kalton's property.
\end{proposition}

\begin{proof}
\ Let us assume that the assertion is not true.
 So there exist $\delta>0$ and $\varepsilon>0$ such that the conclusion does not hold, namely, for every finite set $E \subset M$, there is $\mu\in W$ such that
 $$ d(\mu,\Free([E]_\delta)) > \ep.$$
 We may assume that $R\delta^{-1}\geq 1$ where $R=\rad(M)$. We will first construct sequences $(\mu_n)_n \subset W$  and $(\lambda_n)_n \subset \Free(M)$ such that $\norm{\mu_n-\lambda_n}\to 0$ and
 every $\lambda_n$ is finitely supported.
 Let $\mu_1 \in W$ be arbitrary such that $\norm{\mu_1}>\varepsilon$. 
 Let $\lambda_1 \in \Free(M)$ be such that $\norm{\mu_1-\lambda_1}\leq 2^{-1}\varepsilon$ and $E_1=\supp(\lambda_1)$ is finite.
 By the hypothesis there exists $\mu_2 \in W$ such that $d(\mu_2,\Free([E_1]_\delta)>\varepsilon$.
 Let $\lambda_2 \in \Free(M)$ be finitely supported and such that $\norm{\mu_2-\lambda_2}\leq 2^{-2}\varepsilon$. We denote $E_2=E_1\cup\supp(\lambda_2)$.
 Notice that $d(\lambda_2,[E_1]_\delta)>\varepsilon/2$.
 Continuing this way we will get an increasing family of finite sets $(E_n)_n$ and a sequence $(\lambda_n)_n \subset \Free(M)$ such that for $n \in \N$:
 \begin{itemize}
     \item \ $\lambda_n \in \Free(E_n)$,
     \item \ $\|\lambda_n - \mu_n\| \leq 2^{-n} \ep$ and
     \item \ $d(\lambda_n,\Free([E_{n-1}]_\delta))>\varepsilon/2.$
 \end{itemize}
  Now since $W$ is weakly precompact, there is a weakly Cauchy subsequence $(\mu_{n_k})_k$ of $(\mu_n)_n$.
 By construction $(\lambda_{n_k})_k$ is also weakly Cauchy. 
 Notice that since
 $$d(\lambda_{n_k},\Free([E_{n_k-1}]_\delta))>\varepsilon/2,$$
 we also have $d(\lambda_{n_k},\Free([E_{n_{k-1}}]_\delta))>\varepsilon/2$. 
 So we may write $(\lambda_n)_n$ instead of $(\lambda_{n_k})_k$ in the sequel.
 Denote $\xi_n:=\lambda_{n}-\lambda_{n-1}$, so that $(\xi_n)_n$ is weakly null.
 We have $\supp(\xi_n) \subset E_n$ and $d(\xi_n,\Free([E_{n-1}]_\delta))=d(\lambda_n,\Free([E_{n-1}]_\delta))>\varepsilon/2$.
 Here we assume tacitly that $\lambda_0=0$ and $E_0=\set{0}$.
 
 We will show that this leads to a contradiction.
Let $n_1=1$ and choose a positive $h_1\in B_{\Lip_0(M)}$ such that $\abs{\duality{h_1,\xi_{n_1}}}>\varepsilon/4$. By induction, we will build an increasing sequence of integers $(n_k)_{k \in \N}$ as well as Lipschitz maps $(h_k)_{k \in \N} \subset R\delta^{-1}B_{\Lip_0(M)}$ with the following properties:
\begin{itemize} 
		\item \ $|\duality{h_k,\xi_{n_k}}| \geq \ep/4$,
		\item \ $\abs{\duality{h_i,\xi_{n_j}}}<\frac{\varepsilon}{8k}$ whenever $i\leq k$ and $j>k$,
		\item \ $h_{k} \geq 0$, $h_{k}$ is zero on $[E_{n_{k-1}}]_\delta$ and $(h_k)_k$ have mutually disjoint supports.
\end{itemize}
Assume that $n_k$ has been selected, as well as $h_1,\ldots,h_k$. 
We pick $n_{k+1}>n_k$ such that for all $n\geq n_{k+1}$ and all $i\leq k$ we have $\abs{\duality{h_i,\xi_n}}<\dfrac{\varepsilon}{8k}$.
By the Hahn-Banach theorem there is $f_{k+1}\in B_{\Lip_0(M)}$ which is zero on $[E_{n_k}]_\delta$ and  $\duality{f_{k+1},\xi_{n_{k+1}}}>\varepsilon/2$.
We define $g_{k+1}$ as either the positive or the negative part of $f_{k+1}$. 
The choice is made so that $\abs{\duality{g_{k+1},\xi_{n_{k+1}}}} \geq \varepsilon/4$. 
We now define
\[
h_{k+1}(x) = \max\Big( \sup_{y\in S} (g_{k+1}(y)-R\delta^{-1}d(x,y))  \; , \; 0 \Big),
\]
for $x\in M$, where $S = \supp(\xi_{n_{k+1}})$.
Notice that $h_{k+1}$ is the smallest positive $\frac{R}{\delta}$-Lipschitz extension to the whole space $M$ of the restriction $g_{k+1}\restricted_{S}$. Since $g_{k+1}$ is also clearly one such extension, we have $$0\leq h_{k+1}\leq g_{k+1}.$$ 
In particular, $h_{k+1}$ is zero on $[E_{n_k}]_\delta$ (and everywhere where $g_{k+1}$ was zero).
Now if $l<k$ and $h_{k}(x)\neq 0$ then $x \notin [E_{n_l}]_\delta$. But then $h_l(x)=0$ by the choice of the constant of the extension above.
So $(h_k)_k$ have mutually disjoint supports. This finishes the inductive construction.

Now, letting $h=\sup_k h_k$ we have $h=\sum_{k=1}^\infty h_k$ pointwise and $\lipnorm{h}\leq R\delta^{-1}$. Therefore $h\in\Lip_0(M)$. To finish, for every $k$ we have
\begin{align*}
\abs{\duality{\xi_{n_k},h}} = \abs{\duality{\xi_{n_k},\sum_{i=1}^k h_i}} &\geq \abs{\duality{\xi_{n_k},h_k}} - \sum_{i=1}^{k-1}\abs{\duality{\xi_{n_k},h_i}} \\
&\geq \frac{\varepsilon}{4} - \sum_{i=1}^{k-1}\frac{\varepsilon}{8(k-1)} \geq \frac{\varepsilon}{8}
\end{align*}
contradicting the fact that $(\xi_n)_n$ is weakly null.
\end{proof}

Notice that Theorem~\ref{thm:WeaklyPrecompactIsTight} is now proved for bounded metric spaces. 
In order to remove the hypothesis of boundedness in Proposition~\ref{prop:generalKaltonLemma}, we need to undergo some more tedious work.

\begin{lemma}\label{lm:fuck_boundedness} 
Let $M$ be a pointed metric space and let $(\mu_n)_n$ be a weakly Cauchy sequence in $\lipfree{M}$. 
Then for every $\varepsilon>0$ there exists a bounded set $C\subset M$ such that
\[
(\mu_n)_n\subset \Free(C) + \varepsilon B_{\Free(M)}.
\]
\end{lemma}

\begin{proof}
\ Without loss of generality, we may assume that all the $\mu_n$ are finitely supported.

We will first prove the lemma under the assumption that $(\mu_n)_n$ is weakly null. 
Aiming for a contradiction, suppose that the lemma fails, that is, there exists $\ep>0$ such that for every bounded subset $C \subset M$:
$$ \sup_n d(\mu_n , \Free(C)) > \ep.$$
Since every $\mu_n$ has bounded support, it follows easily that we may, in fact, replace ``$\sup$'' by ``$\limsup$'' in the inequality above.
Let $R_0=1$, $n_0=1$ and $g_0=0$. By induction, we will construct sequences $(n_k)_k$ in $\N$, $(R_k)_k$ in $\R$ 
and $(g_k)_k$ in $\Lip_0(M)$ with the following properties:
\begin{itemize}
	\item \ $\abs{\duality{\mu_{n_k},g_k}}\geq\frac{\varepsilon}{2}$,
	\item \ $\abs{\duality{\mu_{n_k} ,g_i}}\leq 2^{-(2+i)}\varepsilon$ whenever $i \leq k-1$,
	\item \ $g_k\geq 0$, $\lipnorm{g_k}\leq 3$ and $g_k$ vanishes on $B(0,2R_{k-1})$ but is supported on $B(0,2R_k)$,
	\item \ $\mu_{n_k}\in\lipfree{B(0,R_k)}$ and $R_k>2R_{k-1}$.
\end{itemize}
Suppose the sequences have been defined up to index $k-1$.
Since $(\mu_n)_n$ is weakly null and by assumption, we can choose $n_k>n_{k-1}$ such that
\begin{equation}
\label{eq:lemma_fuck_boundedness}
d(\mu_{n_k},\lipfree{B(0,2R_{k-1})})\geq\varepsilon
\end{equation}
and such that
$$\abs{\duality{\mu_{n_k} ,g_i}}\leq 2^{-(2+i)}\varepsilon$$
for $i\leq k-1$.
By the Hahn-Banach theorem, there is $f_k\in\ball{\Lip_0(M)}$ that vanishes on $B(0,2R_{k-1})$ and such that $\duality{\mu_{n_k},f_k}\geq\varepsilon$.
By replacing $f_k$ with its positive or negative part, we may assume that $f_k$ is positive and $\abs{\duality{\mu_{n_k},f_k}}\geq\frac{\varepsilon}{2}$ instead.
Now let $R_k=\rad(\supp(\mu_{n_k}))$ (notice that $R_k>2R_{k-1}$ by \eqref{eq:lemma_fuck_boundedness}) and let $g_k=f_k\cdot h_k$ where
$$
h_k(x)=\begin{cases}
1 &,\; d(x,0)\leq R_k \\
2-\frac{d(x,0)}{R_k} &,\; R_k\leq d(x,0)\leq 2R_k \\
0 &,\; d(x,0)\geq 2R_k
\end{cases}
$$
for $x\in M$. Then $g_k\geq 0$, $\abs{\duality{\mu_{n_k},g_k}}=\abs{\duality{\mu_{n_k},f_k}}\geq\frac{\varepsilon}{2}$, $\supp(g_k)\subset B(0,2R_k)$, and $\lipnorm{g_k}\leq 3\lipnorm{f_k}\leq 3$ by \eqref{mult_operator_norm}, since $\norm{T_{h_k}}\leq 3$. This completes the construction.

Now let $g=\sup_k g_k$, which also equals the pointwise sum of the $g_k$. Then $\lipnorm{g}\leq 3$ so $g\in\Lip_0(M)$. For every $k$ we have
\begin{align*}
\abs{\duality{\mu_{n_k},g}} = \abs{\duality{\mu_{n_k},\sum_{i=1}^k g_i}} &\geq \abs{\duality{\mu_{n_k},g_k}} - \sum_{i=1}^{k-1}\abs{\duality{\mu_{n_k},g_i}} \\
&\geq \frac{\varepsilon}{2} - \sum_{i=1}^{k-1}\frac{\varepsilon}{2^{2+i}} \geq \frac{\varepsilon}{4}
\end{align*}
contradicting the fact that $(\mu_n)$ is weakly null. This settles the weakly null case.

In the general case where $(\mu_n)_n$ is weakly Cauchy, again assume for contradiction that the lemma fails. We may again extract a subsequence $(\mu_{n_k})_k$ such that
$$d(\mu_{n_k}, \Free(B(0,2R_{k-1}))) > \ep ,$$
where $R_k = \rad(\supp(\mu_{n_k}))$. We let $\gamma_k = \mu_{n_k} - \mu_{n_{k-1}}$ for every $k \geq 2$. 
It is readily seen that $(\gamma_k)_k$ is weakly null. Moreover, since $\mu_{n_k} \in \Free(B(0,R_{k}))$ for every $k$, we also deduce that
$$ d(\gamma_k , \Free(B(0,2R_{k-1}))) = d(\mu_{n_k} , \Free(B(0,2R_{k-1})) )  > \ep. $$
Notice that the support of $\mu_{n_k}$ is not contained in the ball centered at $0$ and of radius $2R_{k-1}$, in other words $R_k\geq 2R_{k-1}$. Therefore $\lim\limits_{k \to +\infty} R_k = +\infty$, this contradicts the first part of the proof.
\end{proof}

\begin{lemma}\label{lm:fuck_boundedness2} 
Let $M$ be a pointed metric space and let $W \subset \lipfree{M}$ 
be weakly precompact.
Then for every $\varepsilon>0$ there exists a bounded set $C\subset M$ such that
\[
W\subset \Free(C) + \varepsilon B_{\Free(M)}.
\]
\end{lemma}
 \begin{proof} 
 \ Aiming for a contradiction, assume that the conclusion is not true for some fixed $\varepsilon>0$. Let $x_1 \in W$ such that $\|x_1\|> \ep$ and let us write $R_1 = 0$. Using our assumption we may build by induction an increasing sequence $(R_n)_n \subset \Real$ and a sequence $(x_n)_n \subset W$ such that $\lim\limits_{n \to \infty} R_n = \infty$ and for every $n \in \Natural$
\begin{equation}
\label{eq:lemma_fuck_boundedness2}
d\Big(x_{n} , \Free\big(B(0,R_{n-1})\big)\Big) > \ep.
\end{equation}
Now since $(x_n)_n\subset W$, it admits a weakly Cauchy subsequence $(x_{n_k})_k$.
But this together with \eqref{eq:lemma_fuck_boundedness2} contradicts Lemma~\ref{lm:fuck_boundedness}.
 \end{proof}

We now prove the unbounded version of Proposition~\ref{prop:generalKaltonLemma}.
\begin{proposition}[General case] \label{prop:GeneralKaltonLemma} Let $M$ be a pointed metric space. 
 If $W \subset \Free(M)$ is a weakly precompact set, then $W$ has Kalton's property.
\end{proposition}

 \begin{proof}
 \ Fix $\ep, \delta>0$ and let $W$ be a weakly precompact set in $\Free(M)$. By Lemma~\ref{lm:fuck_boundedness2}, there is a bounded set $C \subset M$ such that $W \subset \Free(C) +  \ep  B_{\Free(M)}$. Without loss of generality, assume that $C=B(0,R)$ for some $R>0$ and denote $C'=B(0,2R)$. Next, define a map $h :M \to \R$ as follows: $h(x) = 1$ for every $ x \in C$, $h(x) = 0$ whenever $x \not \in C'$ and extend $h$ on $C'\setminus C$ using the McShane-Whitney extension formula so that $\lipnorm{h}\leq\frac{1}{R}$ (plus bounding above and below by 1 and 0 to ensure $0 \leq h(x) \leq 1$ for every $x$).
Let $T_h : \Lip_0(C') \to \Lip_0(M)$ be the linear operator defined in \eqref{mult_operator}, and write $S_h : \Free(M) \to \Free(C')$ for its predual operator. Note that
$\norm{S_h}=\norm{T_h}\leq 3$ by \eqref{mult_operator_norm} and that $S_h$ acts on $\lipfree{C}$ as the identity.
Since $W$ is weakly precompact, the set $S_h(W)$ is also weakly precompact in $\Free(C')$. 
By Proposition~\ref{prop:generalKaltonLemma}, there exists a finite set $E \subset C'$ such that
$$S_h(W) \subset \Free([E]_{\delta}) + \ep B_{\Free(M)}.$$

Consider $x \in W$. Since $W \subset \Free(C) + \ep B_{\Free(M)}$, there exists $y \in \Free(C)$ such that $\|x-y\| \leq  \ep$. Consequently we have
$$  \| x - S_h(x) \| \leq \|x - y\| + \|S_h(y) - S_h(x)\| \leq 4\ep . $$
We now conclude with the following inclusions:
$$W \subset S_h(W) + 4\ep B_{\Free(M)} \subset \Free([E]_{\delta}) + 5\ep B_{\Free(M)}.$$
\end{proof}

Proposition~\ref{prop:GeneralKaltonLemma} together with Theorem~\ref{thm:KaltonPropertyToCompact} now prove Theorem~\ref{thm:WeaklyPrecompactIsTight} as promised.

\section{A remark about the Schur property in free spaces}\label{s:SchurExamples}

It is proved in \cite[Proposition 8]{Petitjean} that if the set of uniformly locally flat functions in $\Lip_0(M)$, i.e. 
\[\mathrm{lip}_0(M)=\set{f \in \Lip_0(M): \lim_{\varepsilon\to 0} \sup_{0<d(x,y)<\varepsilon} \frac{\abs{f(x)-f(y)}}{d(x,y)}=0},\] 
is 1-norming for $\F M$ then $\F M$ has the Schur property. This result is in fact an extension of a previous result due to Kalton \cite[Theorem 4.6]{Kalton04}. The fact that $\mathrm{lip}_0(M)$ is norming with constant 1 is essential in the proof of those results. However, thanks to a renorming trick (see also \cite[p. 150]{Weaver2}) we can slightly generalise the result. 

\begin{proposition}
Let $M$ be a metric space such that $\mathrm{lip}_0(M)$ is $C$-norming for some $C\geq 1$, that is 
\[ \forall \gamma \in \F M , \quad \|\gamma \| \leq C \sup_{f\in B_{\mathrm{lip}_0(M)}} | \langle f , \gamma \rangle | . \]
Then $\F M$ has the Schur property. 
\end{proposition}

\begin{proof}
\ Let us define a new metric $\Tilde{d}$ on $M$ by the following formula
\[\forall x,y \in M, \quad \Tilde{d}(x,y) = \sup_{f\in B_{\mathrm{lip}_0(M)}} \big| \langle f , \delta(x) - \delta(y) \rangle \big| .\]
We first notice that $\Tilde{d}$ is equivalent to $d$. Indeed,
\begin{align*}
\Tilde{d}(x,y) &\leq \sup_{f\in B_{\mathrm{Lip}_0(M)}} | \langle f , \delta(x) - \delta(y) \rangle |  =  d(x,y) \\
&\leq C \sup_{f\in B_{\mathrm{lip}_0(M)}} | \langle f , \delta(x) - \delta(y) \rangle | = C \Tilde{d}(x,y).
\end{align*}
Notice that $\ball{\Lip_0(M,\Tilde{d})}\subset\ball{\Lip_0(M,d)}$ since $\Tilde{d}\leq d$. But moreover $\ball{\lip_0(M,d)}\subset\ball{\lip_0(M,\Tilde{d})}$ by the definition of $\Tilde{d}$, so we actually have $\ball{\lip_0(M,d)}=\ball{\lip_0(M,\Tilde{d})}$. Next, it is clear 
from the definition that for every $x \neq y$ and every $\ep>0$, there exists $f \in B_{\mathrm{lip}_0(M)}$ such that $|f(x) - f(y)| \geq \Tilde{d}(x,y) - \ep$ (we say that $\mathrm{lip}_0(M,\Tilde{d})$ 1--separates points of $(M,\Tilde{d})$ uniformly). Since $\mathrm{lip}_0(M,\Tilde{d})$ is a sublattice of $\Lip_0(M,\Tilde{d})$, \cite[Proposition 3.4]{Kalton04} implies that
$\mathrm{lip}_0(M,\Tilde{d})$ is $1$-norming for $\F{ M,\Tilde{d}}$.
According to \cite[Proposition 8]{Petitjean}, $\F{ M,\Tilde{d}}$ has the Schur property. The conclusion follows from the fact that $\F{ M,\Tilde{d}}$ and $\F M$ are isomorphic.
\end{proof}

It is quite natural to ask whether this last condition is equivalent to the Schur property. Not very surprisingly, it is not. Our first example shows that it might happen that $\mathrm{lip}_0(M)$ does not even separate the points of $\Free(M)$ while the latter space has the Schur property. The second example shows that assuming that $M$ is topologically discrete would not help either.

\begin{example} \label{ExampleSchur}
There exists a countable complete metric space $M$ such that $\mathrm{lip}_0(M)$ does not separate points of $\F M$ and such that $\F M$ has the Schur property. 
\end{example}
\begin{proof}
\ We first define a metric graph structure as follows:
$$M = \{0,q\}\cup \{x_i^n \; : \; n \in \mathbb N, 1 \leq i \leq n \},$$ with edges $(0,x_1^n)$, $(x_i^n , x_{i+1}^n)$ and $(x_n^n,q)$. 
The metric $d$ is defined on the edges by 
\[d(0,x_1^n) = d(x_i^n , x_{i+1}^n) = d(x_n^n,q) = \dfrac{1}{n+1}, \]
and then extended as the shortest path distance along the edges. 
Note that $d(0,q) = 1$ and that $0$ and $q$ are the only accumulation points of $M$.

The fact that $\F M$ has the Schur property follows from Corollary~\ref{c:FCountableCompletIsSchur} as $M$ is countable and complete. Let us now check that $\mathrm{lip}_0(M)$ does not separate points of $\F M$. 
Indeed, aiming at a contradiction, assume that there exists $f \in \mathrm{lip}_0(M) $ such that $f(q) = 1$. 
For every $\varepsilon >0$, there is $\delta(\varepsilon) >0$ such that 
\[ \forall x,y \in M , \quad d(x,y) \leq \delta(\varepsilon) \implies |f(x) - f(y) | \leq \varepsilon d(x,y). \]
Let $n$ be such that $\frac{1}{n+1} < \delta(\frac{1}{2})$. 
Then we have for every $1\leq i \leq n$  that
$$|f(x_i^n) - f(x_{i+1}^n)| \leq \frac{1}{2}\, \dfrac{1}{n+1}.$$  
Thus 
\[ f(q) \leq |f(q) - f(x_n^n) | + \sum_{i=1}^{n-1} |f(x_i^n) - f(x_{i+1}^n)| + |f(x_1^n) - f(0) |\leq  \dfrac{1}{2},\]
which is clearly a contradiction.
\end{proof}

The following example is similar.
\begin{example}
There exists a countable, topologically discrete, complete metric space $M$ such that $\mathrm{lip}_0(M)$ is not norming for $\F M$ and such that $\F M$ has the Schur property. 
\end{example}
\begin{proof}
\ The metric space $M$ is defined as the disjoint union, at a suitably large distance from each other, of metric graphs $M_n$ where $n \in \mathbb N$ and $M_n$ is defined by 
$$M_n = \{p_n,q_n\}\cup \{x_i^N \; : \; N \in \mathbb N \setminus \{1\}, 1 \leq i \leq N \},$$ with edges $(p_n,x_1^N)$, $(x_i^N , x_{i+1}^N)$ and $(x_N^N,q_n)$. The metric $d$ is defined on the edges by
\[d(p_n,x_1^N) = d(x_N^N,q_n) = \dfrac{1}{n}, \quad  d(x_i^N , x_{i+1}^N) = \dfrac{1}{N-1}, \]
and then extended on each $M_n$ as the shortest path distance along the edges. 
Note that $d(p_n,q_n) = \frac{2}{n}+1$. Let us fix $p_1$ to be the base point of $M$. 

The fact that $\Free(M)$ has the Schur property follows again directly from Corollary~\ref{c:FCountableCompletIsSchur}.
Let us see that $\lip_0(M)$ cannot be $C$-norming for any $C<\infty$. Indeed, fix $C$, choose $n>2C$ and suppose that $f\in\lip_0(M)$ is such that $f(q_n)-f(p_n)=d(p_n,q_n)=1+\frac{2}{n}$. Fix $\varepsilon>0$, then there is $\delta>0$ such that $\abs{f(x)-f(y)}\leq\varepsilon d(x,y)$ whenever $d(x,y)\leq\delta$. Take $N$ such that $\frac{1}{N-1}<\delta$. Then we have
\begin{align*}
f(q_n)-f(p_n) &\leq \abs{f(q_n)-f(x_N^N)} + \sum_{i=1}^{N-1}\abs{f(x_{i+1}^N)-f(x_i^N)} + \abs{f(x_1^N)-f(p_n)} \\
&\leq \frac{1}{n}\lipnorm{f} + (N-1)\cdot\frac{\varepsilon}{N-1} + \frac{1}{n}\lipnorm{f} = \frac{2}{n}\lipnorm{f}+\varepsilon .
\end{align*}
Since $\varepsilon>0$ was arbitrary, we get $1+\frac{2}{n}\leq\frac{2}{n}\lipnorm{f}$ and hence $\lipnorm{f}\geq 1+\frac{n}{2}>C$. Thus $\lip_0(M)$ is not $C$-norming.
\end{proof}

\subsection*{Acknowledgements}

This work was supported by the French ``Investissements d'Avenir'' program, project ISITE-BFC (contract ANR-15-IDEX-03). R. J. Aliaga was partially supported by the Spanish Ministry of Economy, Industry and Competitiveness under Grant MTM2017-83262-C2-2-P.

\end{document}